\documentclass[12pt]{amsart}
\usepackage{amssymb}
\newtheorem{theorem}{Theorem}
\newtheorem{lemma}[theorem]{Lemma}

\title{On closures of discrete sets}

\author{Santi Spadaro}

\address{Department of Mathematics and Computer Science \\
University of Catania \\
Citt\'a universitaria\\  
viale A. Doria 6 \\
95125 Catania, Italy}

\email{santidspadaro@gmail.com}



\begin{document}
\maketitle

\begin{abstract} 
The depth of a topological space $X$ ($g(X)$) is defined as the supremum of the cardinalities of closures of discrete subsets of $X$. Solving a problem of Mart\'inez-Ruiz, Ram\'irez-P\'aramo and Romero-Morales, we prove that the cardinal inequality $|X| \leq g(X)^{L(X) \cdot F(X)}$ holds for every Hausdorff space $X$, where $L(X)$ is the Lindel\"of number of $X$ and $F(X)$ is the supremum of the cardinalities of the free sequences in $X$. 
\end{abstract}

\section{Introduction}
We first witnessed the strong influence of discrete sets on cardinal properties of topological spaces in De Groot's 1965 paper \cite{Gr}, where he proved that the cardinality of a regular space without uncountable discrete subsets is bounded by $2^{2^{\aleph_0}}$. Hajnal and Juh\'asz \cite{HJ} later improved De Groot's bound by replacing the regularity assumption with the Hausdorff property and proved the important result that the cardinality of a $T_1$ space with points $G_\delta$ and without uncountable discrete sets is bounded by the continuum. 

The \emph{depth} of a space $X$ ($g(X)$) is defined as the supremum of the cardinalities of closures of discrete subsets of $X$. Arhangel'skii \cite{ArDepth} asked whether $g(X)=|X|$ for every compact space $X$. Dow \cite{DowDis} found consistent counterexamples to Arhangel'skii's question and proved that $|X| \leq g(X)^\omega$ for every compact space $X$ of countable tightness. He also considered the following stronger property: a space $X$ is called \emph{cardinality discretely reflexive} if there is a discrete set $D \subset X$ such that $|\overline{D}|=|X|$. In \cite{DowDis} Dow asked whether the classes of compact separable spaces and of compact spaces of countable tightness are discretely reflexive. Some partial answers to these questions were published in \cite{JS} and \cite{SIncreasing}. 

A few more bounds on the cardinality of a topological space by means of the depth can be found in \cite{Al}, \cite{SDiscrete} and \cite{RG}. For example, the authors of \cite{RG} proved that $|X| \leq g(X)^{L(X) \cdot t(X)}$ for every Hausdorff space $X$, where $L(X)$ denotes the Lindel\"of degree of $X$ and $t(X)$ denotes the tightness of $X$. 

Recall that a well-ordered set $\{x_\alpha: \alpha < \kappa \}$ is called a \emph{free sequence} if, for every $\beta < \kappa$, we have $\overline{\{x_\alpha: \alpha < \beta\}} \cap \overline{\{x_\alpha: \alpha \geq \beta \}}=\emptyset$. The supremum of cardinalities of free sequences in $X$ is denoted by $F(X)$.

It is well-known and easy to prove that $F(X) \leq L(X) \cdot t(X)$, and a famous result of Arhangel'skii states that $F(X)=t(X)$ for every compact space $X$. On the other hand there are examples of even $\sigma$-compact spaces where the gap between $F(X)$ and $t(X)$ can be arbitrarily large (see \cite{Ok}). It is then natural to ask whether the bound $|X| \leq g(X)^{F(X) \cdot L(X)}$ is true for every Hausdorff space $X$. This is Problem 3.17 in \cite{MRR}. We present a solution to this problem.

In the proof we will make use elementary submodels of the structure $(H(\theta), \epsilon)$. Dow's survey \cite{D} and Fedeli and Watson's paper \cite{FW} provide a good introduction to the use of this technique in topology. We now recall the basic definitions and state, without proof, the results about elementary submodels we need in our paper. Recall that $H(\theta)$ is the set of all sets whose transitive closure has cardinality smaller than $\theta$. When $\theta$ is a regular uncountable cardinal, the set $H(\theta)$ is known to satisfy all axioms of set theory, except the power set axiom. We say, informally, that a formula is satisfied by a set $S$ if it is true when all bounded quantifiers are restricted to $S$. A subset $M$ of $H(\theta)$ is said to be an elementary submodel of $H(\theta)$ (and we denote that by $M \prec H(\theta)$) if a formula with parameters in $M$ is satisfied by $H(\theta)$ if and only if it is satisfied by $M$. 

The downward L\"owenheim-Skolem theorem guarantees that for every set $S \subset H(\theta)$, there is an elementary submodel $M \prec H(\theta)$ such that $|M| \leq |S| \cdot \omega$ and $S \subset M$. It is often useful for $M$ to have the following closure property: we say that $M$ is $\kappa$-closed if for every $S \subset M$ such that $|S| \leq \kappa$ we have $S \in M$. Let $\theta$ be a regular cardinal larger than $\mu^\kappa$. For every set $S \subset H(\theta)$ of cardinality $\leq \mu^\kappa$ there is always a $\kappa$-closed elementary submodel $M \prec H(\theta)$ such that $|M| \leq \mu^{\kappa}$ and $S \subset M$. 

The following fact is also used often: let $M$ be an elementary submodel of $H(\theta)$ such that $\mu + 1 \subset M$ and let $S \in M$ be a set of cardinality $\leq \mu$. Then $S \subset M$. 

All spaces under consideration are assumed to be Hausdorff. Undefined notions can be found in \cite{E} for topology and \cite{Ku} for set theory. Our notation regarding cardinal functions mostly follows \cite{J}. In particular, $L$, $\chi$, $\psi$ and $t$ denote the Lindel\"of degree, the character, the pseudocharacter and the tightness respectively.

\section{The main result}

We make liberal use of the following easy to prove lemma, often attributed to Shapirovskii (see \cite{J}).

\begin{lemma}
(Shapirovskii) Let $\mathcal{U}$ be an open cover of the topological space $X$. Then there is a discrete set $D \subset X$ and a subcollection $\mathcal{V} \subset \mathcal{U}$ such that $|\mathcal{V}|=|D|$ and $X= \overline{D} \cup \bigcup \mathcal{V}$.
\end{lemma}

The following theorem solves Problem 3.17 from \cite{MRR}.

\begin{theorem} \label{mainthm}
Let $X$ be a Hausdorff space. Then $|X| \leq g(X)^{L(X) \cdot F(X)}$.
\end{theorem}

\begin{proof}
Let $\kappa=L(X) \cdot F(X)$, let $\mu=g(X)$ and let $M$ be an elementary submodel of $H(\theta)$ (where $\theta$ is a large enough regular cardinal) such that $M$ is $\kappa$-closed, $X \in M$ $\mu^\kappa+1 \subset M$ and $|M| \leq \mu^\kappa$.

\vspace{.1in}

\noindent \textbf{Claim 1.} $\psi(X) \leq \mu$.

\begin{proof}[Proof of Claim 1]
Fix a point $x \in X$. For every point $y \neq x$, choose an open neighbourhood $U_y$ of $y$ such that $x \notin \overline{U_y}$. Use Shapirovskii's Lemma to find a discrete set $D \subset X \setminus \{x\}$ such that $X \setminus \{x\} \subset \bigcup \{U_y: y \in D \} \cup \overline{D}$. Then $\{X \setminus \overline{U}_y: y \in D \} \cup \{X \setminus \{y\}: y \in \overline{D} \setminus \{x\} \}$ is a family of open sets whose intersection is $\{x\}$.
\renewcommand{\qedsymbol}{$\triangle$}
\end{proof}

\noindent \textbf{Claim 2.}  Let $Y$ be a subspace of $X$ such that $F(Y) \leq \kappa$ and with the property that the closure of every free sequence in $Y$ has Lindel\"of number at most $\kappa$. Let $\mathcal{U}$ be an open cover of $Y$. Then there is a free sequence $F \subset Y$ and a subcollection $\mathcal{V} \subset \mathcal{U}$ such that $|\mathcal{V}| \leq \kappa$ and $Y \subset \overline{F} \cup \bigcup \mathcal{V}$.

\begin{proof}[Proof of Claim 2]
Suppose you have constructed, for some ordinal $\beta<\kappa^+$, a free sequence $\{x_\alpha: \alpha < \beta \}$ and $\kappa$-sized subcollections $\{\mathcal{U}_\alpha: \alpha < \beta \}$ of $\mathcal{U}$ such that $\overline{\{x_\alpha : \alpha < \gamma\}} \subset \bigcup \bigcup_{\alpha \leq \gamma} \mathcal{U}_\alpha$ for every $\gamma < \beta$.

Let $\mathcal{U}_\beta$ be a $\kappa$-sized subcollection of $\mathcal{U}$ which covers the subspace $\overline{\{x_\alpha : \alpha < \beta \}}$. If $\bigcup_{\alpha \leq \beta} \mathcal{U}_\alpha$ covers $Y$ then we are done; otherwise, pick a point $x_\beta \in Y \setminus \bigcup \bigcup_{\alpha \leq \beta} \mathcal{U}_\alpha$. Let $\tau$ be the least ordinal such that $$Y \subset \overline{\{x_\alpha: \alpha < \tau\}} \cup \bigcup \bigcup_{\alpha < \tau} \mathcal{U}_\tau.$$ Obviously, $\tau$ must be less than $\kappa^+$, or otherwise we would get a free sequence of cardinality larger than $\kappa$ in $Y$. Then $F=\{x_\alpha : \alpha < \tau\}$ and $\mathcal{V}=\bigcup_{\alpha < \tau} \mathcal{U}_\alpha$ satisfy the statement of the claim.
\renewcommand{\qedsymbol}{$\triangle$}
\end{proof}

\noindent \textbf{Claim 3.} Let $Y=X \cap M$. Then $Y$ satisfies the hypotheses of Claim 2.

\begin{proof}[Proof of Claim 3]
First of all we prove that every free sequence in $X \cap M$ is also a free sequence in $X$. Indeed, let $F=\{x_\alpha: \alpha < \lambda \} \subset X \cap M$ be a free sequence (in the relative topology on $X \cap M$). Let $F_\beta=\{x_\alpha: \alpha < \beta \}$ and set $\alpha=\sup \{\beta < \lambda: F_\beta$ is a free sequence in $X$ by the same well-ordering of $F \}$.

Then $F_\alpha$ is a free sequence in $X$. If not, there would be some ordinal $\beta < \alpha$ and a point $x \notin M$ such that $x \in \overline{F_\beta} \cap \overline{F_\alpha \setminus F_\beta}$. But $F_\beta$ is a free sequence in $X$ and therefore $|F_\beta| \leq \kappa$.  Since $M$ is $\kappa$-closed we have $F_\beta \in M$, and hence $\overline{F_\beta} \in M$. But that and $|\overline{F_\beta}| \leq \mu$ imply that $\overline{F_\beta} \subset M$. So $x \in M$, which is a contradiction. 

It follows that $F_\alpha$ is a free sequence in $X$, but then we must have $\alpha=\lambda$, or otherwise $F_{\alpha+1}$ would also be a free sequence in $X$, contradicting the definition of $\alpha$. But $F_\lambda=F$, so $F$ is a free sequence in $X$. 

A consequence of what we just proved is that $F(X \cap M) \leq \kappa$. Let now $F \subset X \cap M$ be a free sequence, then $|F| \leq \kappa$ and hence $F \in M$, by $\kappa$-closedness of $M$. But then $\overline{F} \in M$ and since $|\overline{F}| \leq \mu$ we must have $\overline{F} \subset X \cap M$. Since $X$ has Lindel\"of number $\leq \kappa$, it follows that $X \cap M$ has the property that the closure of all its free sequences have Lindel\"of number at most $\kappa$, which is what we wanted to prove.

\renewcommand{\qedsymbol}{$\triangle$}
\end{proof}

We claim that $X \subset M$. Suppose not, and let $p \in X \setminus M$. For every $x \in X \cap M$ use Claim 1 to pick an open neighbourhood $U_x \in M$ of $x$ such that $p \notin U_x$. Let $\mathcal{U}=\{U_x: x \in X \cap M \}$, which is an open cover of $X \cap M$. By Claim 2 there are a free sequence $F \subset X \cap M$ and a subcollection $\mathcal{V} \subset \mathcal{U}$ such that $|\mathcal{V}| \leq \kappa$ and $X \cap M \subset \overline{F} \cup \bigcup \mathcal{V}$. Now $|F| \leq \kappa$ and $M$ is $\kappa$-closed, so $F \in M$ and hence $\overline{F} \in M$, which, along with $|\overline{F}| \leq \mu$ implies that $\overline{F} \subset M$. Also, $\mathcal{V} \subset M$ and $|\mathcal{V}| \leq \kappa$ imply that $\mathcal{V} \in M$. Therefore $M \models X \subset \overline{F} \cup \bigcup \mathcal{V}$, which, by elementarity, implies that $H(\theta) \models X \subset \overline{F} \cup \bigcup \mathcal{V}$. Since $p \notin \overline{F}$, there must be $V \in \mathcal{V}$ such that $p \in V$, which is a contradiction.

Therefore $|X| \leq |M| \leq \mu^\kappa$, which is what we wanted to prove.
\end{proof}

\end{document}